\documentclass{amsart}

\usepackage{amssymb}

\usepackage{graphicx}

\newcommand{\func}[1]{\operatorname{#1}}
\usepackage{color}
\newtheorem{theorem}{Theorem}[section]

\newtheorem{proposition}[theorem]{Proposition}   % LIDIA

\theoremstyle{definition}

\newtheorem{example}[theorem]{Example}

\theoremstyle{remark}
\newtheorem{remark}[theorem]{Remark}

\numberwithin{equation}{section}

\begin{document}

\title[Approximations to the resolvent of fractional operators]{Pad\'{e}-type approximations to the resolvent of fractional powers of operators}


%    author one information
\author{Lidia Aceto}
\address{Departments of Mathematics, University of Pisa, Via F. Buonarroti, 1/C, 56127 Pisa, Italy}
\curraddr{}
\email{lidia.aceto@unipi.it}
\thanks{The authors are members of the INdAM research group GNCS}

%    author two information
\author{Paolo Novati}
\address{Departments of Mathematics and Geosciences, University of Trieste, via Valerio 12/1, 34127 Trieste, Italy}
\curraddr{}
\email{novati@units.it}
\thanks{This work was supported by GNCS-INdAM and by FRA-University of Trieste}

%    \subjclass is required.

\date{}

\dedicatory{}

%    Abstract is required.
\begin{abstract}
We study a reliable pole selection for the rational approximation of the resolvent of fractional powers of operators in both the finite and infinite dimensional
setting. The analysis exploits the representation in terms of hypergeometric
functions of the error of the Pad\'{e} approximation of the fractional
power. We provide quantitatively accurate error estimates that can be
used fruitfully  for practical computations. We present some numerical
examples to corroborate the theoretical results. The behavior of the
rational Krylov methods based on this theory is also presented.
\end{abstract}

\maketitle


\section{Introduction}

Let $\mathcal{L}$ be a self-adjoint positive operator with spectrum $\sigma (%
\mathcal{L})\subseteq \lbrack c,+\infty ),$ $c>0$, acting on a Hilbert space 
$\mathcal{H}$ endowed with norm $\left\Vert \cdot \right\Vert _{\mathcal{H}}$
and operator norm $\left\Vert \cdot \right\Vert _{\mathcal{H}\rightarrow 
\mathcal{H}}$. We assume that $\mathcal{L}$ possesses a compact inverse so
that it can be written in terms of its spectral decomposition and the
operational calculus $f(\mathcal{L})$ can be defined by working on the
eigenvalues as in the finite dimensional setting. Moreover,  denoting by 
$\{\mu _{j}\}_{j=1}^{\infty }$ the eigenvalues of $\mathcal{L}$ and assuming that they are
numbered in increasing order of magnitude, we have $c=\mu _{1}.$
This paper deals with the numerical approximation of the resolvent of fractional powers%
\[
\left( I+h\mathcal{L}^{\alpha }\right) ^{-1},\quad 0<\alpha <1,\quad h>0,
\]
where $I$ denotes the identity operator. This kind of resolvent appears
for instance when using an implicit multistep or a Runge-Kutta method for
solving fractional in space parabolic-type equations in which $\mathcal{L}$\
represents the Laplacian operator with Dirichlet boundary conditions and $h$ depends  both on the time step and the
parameters of the integrator. We quote here \cite{MMP} and the references therein contained 
for a comprehensive treatment of the operational calculus involving fractional powers, in the 
more generic setting of linear operators on Banach spaces.

Clearly the computation of $\left( I+h\mathcal{L}^{\alpha }\right) ^{-1}$ is closely connected to the approximation of the
fractional power $\mathcal{L}^{-\alpha }$ because%
\begin{equation}
\frac{1}{1+h\lambda ^{\alpha }}=\frac{\lambda ^{-\alpha }}{\lambda ^{-\alpha
}+h},  \label{eqa}
\end{equation}%
and hence, any approximant of the function $\lambda ^{-\alpha }$ in $%
[c,+\infty )$ can be employed to define a method for the resolvent. In this view, 
recalling the analysis given in \cite{AN0}, the
basic aim of this work is to consider Pad\'{e}-type approximations of the
fractional power, centered at points that allows to minimize as much as
possible the error for $\left( I+h\mathcal{L}^{\alpha }\right) ^{-1}$. This
idea is justified by the fact that the function $\left( 1+h\lambda ^{\alpha
}\right) ^{-1}$ behaves like $\lambda ^{-\alpha }$ for large values of $%
\lambda $.

For any given $\tau >0$,\ let $R_{k-1,k}(\lambda /\tau )$ be the $(k-1,k)$%
-Pad\'{e} approximant of $(\lambda /\tau )^{-\alpha }$ centered at $1$, and
consider the approximation 
\begin{equation}
\mathcal{L}^{-\alpha }\approx \mathcal{R}_{k-1,k}(\mathcal{L}),\quad 
\mathcal{R}_{k-1,k}(\lambda ):=\tau ^{-\alpha }R_{k-1,k}(\lambda /\tau ).
\label{pad}
\end{equation}%
It is well known that the choice of the parameter $\tau $ is fundamental for
the quality of the approximation, see \cite{AN0,AN}. Since $\mathcal{L}$ is
assumed to be self-adjoint the analysis of this approach can be made by
working scalarly in the real interval $[c,+\infty )$. In particular, working
with unbounded operator, in \cite{AN0} it has been shown how to suitably
define the parameter $\tau $ by looking for an approximation of the optimal
value given by the solution of%
\[
\min_{\tau >0}\max_{\lambda \in \lbrack c,+\infty )}\left\vert \lambda
^{-\alpha }-\mathcal{R}_{k-1,k}(\lambda )\right\vert .
\]
As for the resolvent, using $\lambda ^{-\alpha }\approx \mathcal{R%
}_{k-1,k}(\lambda )$ in (\ref{eqa}) and writing%
\[
\mathcal{R}_{k-1,k}(\lambda )=\frac{p_{k-1}\left( \lambda \right) }{%
q_{k}\left( \lambda \right) },\quad p_{k-1}\in \Pi _{k-1},q_{k}\in \Pi _{k},
\]
where $\Pi _{j}$ denotes the set of polynomials of degree at most $j$, we
have%
\begin{equation}
\frac{1}{1+h\lambda ^{\alpha }}\approx \frac{p_{k-1}\left( \lambda \right) }{%
p_{k-1}\left( \lambda \right) +hq_{k}\left( \lambda \right) }=:\mathcal{S}%
_{k-1,k}(\lambda ).  \label{sk}
\end{equation}%
Obviously $\mathcal{S}_{k-1,k}$ inherits the dependence on $\tau $, and the
main contribute of this paper is to define the parameter $\tau $ in order to
minimize  as much  as possible the error%
\begin{equation} \label{EK}
E_{k}:=\left\Vert \left( I+h\mathcal{L}^{\alpha }\right) ^{-1}-\mathcal{S}%
_{k-1,k}(\mathcal{L})\right\Vert _{\mathcal{H}\rightarrow \mathcal{H}}
\end{equation}
so that the idea here is to define $\tau $ by looking for the solution of 
\begin{equation}
\min_{\tau >0}\max_{\lambda \in \lbrack c,+\infty )}\left\vert \left(
1+h\lambda ^{\alpha }\right) ^{-1}-\mathcal{S}_{k-1,k}(\lambda )\right\vert .
\label{minm}
\end{equation}%
We derive an approximate solution $\tau _{k}$ of (\ref{minm}) that depends
on $k$ and $h$, and we are able to show that the error $E_{k}$ %qualitatively
decays like $\mathcal{O}(k^{-4\alpha }),$ that is, sublinearly. We experimentally show that using this new parameter
sequence it is possible to improve the approximation attainable by taking $%
\tau _{k}$ as in \cite{AN0} for $\mathcal{L}^{-\alpha }$ and then using (\ref%
{eqa}) to compute $\left( I+h\mathcal{L}^{\alpha }\right) ^{-1}$. The latter
approach has recently been used in \cite{Ber}.

In the applications, where one works with a discretization $\mathcal{%
L}_{N}$ of $\mathcal{L}$, if the largest eigenvalue $\lambda _{N}$ of $%
\mathcal{L}_{N}$ (or an approximation of it) is known, then the theory
developed for the unbounded case can be refined. In particular, here we
present a new sequence of parameters $\left\{ \tau _{k,N}\right\} _{k}$ that
can be use to handle this situation and that allows to compute $\left( I+h%
\mathcal{L}_{N}^{\alpha }\right) ^{-1}$ with a linear decay of the error,
that is, of the type $r^{k}$, $0<r<1$. In both situations, unbounded and bounded, we provide error estimates that
are quantitatively quite accurate and therefore useful for an a-priori
choice of $k$ that, computationally, represents the number of
operator/matrix inversions (cf. (\ref{sk})). 

The poles of $\mathcal{S}_{k-1,k}$ can also be used to define a rational Krylov method for the
computation of $\left( I+h\mathcal{L}_{N}^{\alpha }\right) ^{-1}v$, $v\in 
\mathbb{R}^{N}$, and the error estimates as hints for the a-priori
definition of the dimension of the Krylov space. We remark that the poles of 
$\mathcal{S}_{k-1,k}$ completely change with $k$, so that for a Krylov
method it is fundamental to decide at the beginning the dimension to reach,
that is, the set of poles. The construction of rational Krylov methods based
on the theory presented in the paper is considered at the end of the paper.

The paper is organized as follows. In Section \ref{sec2} the basic features
concerning the rational approximation of the fractional power $\mathcal{L}%
^{-\alpha }$ are recalled. Section \ref{sec3} and Section \ref{sec4} contain
the error analysis for the infinite and finite dimensional case,
respectively. Finally, in Section \ref{sec5} some numerical results are
reported, including some experiments with rational Krylov methods.

%%%%%%%%%%%%%%%%%%%%%%%%%%%%%%%%%%%%%%%%%%%
%%%%%%%%%%%%%%%%%%%%%%%%%%%%%%%%%%%%%%%%%%%

\section{The rational approximation}

\label{sec2}

The $(k-1,k)$-Pad\'{e} type approximation to $\mathcal{L}^{-\alpha }$
recalled in (\ref{pad}) can be obtained starting from the integral
representation (see \cite[Eq. (V.4) p. 116]{RB})
\begin{equation}
\mathcal{L}^{-\alpha }=\frac{\sin (\alpha \pi )}{(1-\alpha )\pi }%
\int_{0}^{\infty }(\rho ^{1/(1-\alpha )}I+\mathcal{L})^{-1}d\rho ,
\label{mat}
\end{equation}%
and using the change of variable%
\begin{equation}
\rho ^{1/(1-\alpha )}=\tau \frac{1-t}{1+t},\qquad \tau >0,  \label{jss}
\end{equation}%
that yields 
\begin{equation}
\mathcal{L}^{-\alpha }=\frac{2\sin (\alpha \pi )\tau ^{1-\alpha }}{\pi }%
\int_{-1}^{1}\left( 1-t\right) ^{-\alpha }\left( 1+t\right) ^{\alpha
-2}\left( \tau \frac{1-t}{1+t}I+\mathcal{L}\right) ^{-1}dt.  \label{nint}
\end{equation}%
Using the $k$-point Gauss-Jacobi rule with respect to the weight function $%
\omega (t)=\left( 1-t\right) ^{-\alpha }\left( 1+t\right) ^{\alpha -1}$ we
obtain the rational approximation (see (\ref{pad})) 
\begin{equation}
\mathcal{L}^{-\alpha }\approx \sum_{j=1}^{k}\gamma _{j}(\eta _{j}I+\mathcal{L%
})^{-1}=\tau ^{-\alpha }R_{k-1,k}\left( \frac{\mathcal{L}}{\tau }\right) =%
\mathcal{R}_{k-1,k}(\mathcal{L}).  \label{rapp}
\end{equation}%
The coefficients $\gamma _{j}$ and $\eta _{j}$ are given by%
\begin{equation}
\gamma _{j}=\frac{2\sin (\alpha \pi )\tau ^{1-\alpha }}{\pi }\frac{w_{j}}{%
1+\vartheta _{j}},\qquad \eta _{j}=\frac{\tau (1-\vartheta _{j})}{%
1+\vartheta _{j}},  \label{gamma_eta}
\end{equation}%
where $w_{j}$ and $\vartheta _{j}$ are, respectively, the weights and nodes
of the Gauss-Jacobi quadrature rule. Denoting by $\zeta _{r}$ the $r$th
zero of the Jacobi polynomial ${\mathcal{P}}_{k-1}^{(\alpha ,1-\alpha
)}\left( \lambda \right) $ and setting 
\begin{equation}
\epsilon _{r}=\tau \frac{1-\zeta _{r}}{1+\zeta _{r}},\quad r=1,2,\dots ,k-1,
\label{eq:epsrad}
\end{equation}%
from~\cite[Proposition~2]{AN2} we can express $\mathcal{R}_{k-1,k}(\lambda )$
as the rational function 
\begin{equation}
\mathcal{R}_{k-1,k}(\lambda )=\frac{p_{k-1}(\lambda )}{q_{k}(\lambda )}=%
\frac{\chi \prod_{r=1}^{k-1}(\lambda +\epsilon _{r})}{\prod_{j=1}^{k}(%
\lambda +\eta _{j})},  \label{quoz}
\end{equation}%
where 
\[
\quad \chi =\frac{\eta _{k}}{\tau ^{\alpha }}\frac{\binom{k+\alpha -1}{k-1}}{%
\binom{k-\alpha }{k}}\prod_{j=1}^{k-1}\frac{\eta _{j}}{\epsilon _{j}}.
\]
We refer here to \cite{Bo, HetP}  for other effective rational approaches based on different integral representations 
and quadrature rules.

As for the resolvent, using the rational form $\mathcal{S}%
_{k-1,k}$ defined in (\ref{sk}) we have that%
\begin{equation}
\left( I+h\mathcal{L}^{\alpha }\right) ^{-1}\approx \sum_{j=1}^{k}\overline{%
\gamma }_{j}(\overline{\eta }_{j}I+\mathcal{L})^{-1},  \label{npad}
\end{equation}%
where $\overline{\gamma }_{j}$ are the coefficients of the partial fraction
expansion of $\mathcal{S}_{k-1,k}$ and $-\overline{\eta }_{j}$ are
the roots of the polynomial $p_{k-1}\left( \lambda \right) +hq_{k}\left(
\lambda \right) \in \Pi _{k}.$ From~\cite[Proposition~1]{Ber} we know that
all the values $-\overline{\eta }_{j}$ are real and simple. To locate them
on the real axis, we recall that $\vartheta _{j}$ are the zeros of the
Jacobi polynomial ${\mathcal{P}}_{k}^{(-\alpha ,\alpha -1)}\left( \lambda
\right) $. So, using (\ref{gamma_eta}) it is immediate to verify that the
roots of $q_{k}(\lambda )$ are all real, simple and negative, which implies
that its coefficients are strictly positive. The same conclusions apply to $%
p_{k-1}(\lambda ).$ Therefore, since all the coefficients of $p_{k-1}\left(
\lambda \right) +hq_{k}\left( \lambda \right) $ are strictly positive by
construction, according to the Descartes' rule of signs, we are also sure
that $-\overline{\eta }_{j}$ are negative and therefore that the
approximation (\ref{npad}) is well-defined.

%%%%%%%%%%%%%%%%%%%%%%%%%%%%%%%%%%%%%%%%%%%
%%%%%%%%%%%%%%%%%%%%%%%%%%%%%%%%%%%%%%%%%%%

\section{Error analysis}

\label{sec3}

Before starting we want to emphasize that the rational forms $\mathcal{R}%
_{k-1,k}$ and $\mathcal{S}_{k-1,k}$ depend on the value of $\tau $ in (\ref%
{jss}). Anyway, in order to keep the notations as simple as possible, in
what follows we avoid to write the explicit dependence on this parameter.
Moreover, throughout this and the following section we frequently use the
symbol $\sim $ to compare sequences, with the underlying meaning that $%
a_{k}\sim b_{k}$ if $a_{k}=b_{k}(1+\varepsilon _{k})$ where $\varepsilon
_{k}\rightarrow 0$ as $k\rightarrow +\infty.$

First of all we recall the following result given in \cite[Proposition 2]%
{AN0}, and based on the representation of the error arising from the Pad\'e approximation 
of the fractional power
\begin{equation}
e_{k}(\lambda ):=\lambda ^{-\alpha }-\mathcal{R}_{k-1,k}(\lambda )
\label{efp}
\end{equation}%
in terms of hypergeometric functions whose detailed analysis can be found in 
\cite{E}.

\begin{proposition}
\label{p1}For large values of $k$, the following representation of the error
holds%
\begin{equation}
e_{k}(\lambda )=2\sin (\alpha \pi )\lambda ^{-\alpha }\left[ \frac{\lambda
^{1/2}-\tau ^{1/2}}{\lambda ^{1/2}+\tau ^{1/2}}\right] ^{2k}\left( 1+%
\mathcal{O}\left( \frac{1}{k}\right) \right) .  \label{ek}
\end{equation}
\end{proposition}

Now, let
\begin{equation} \label{errek}
r_{k}(\lambda ):=\left( 1+h\lambda ^{\alpha }\right) ^{-1}-\mathcal{S}%
_{k-1,k}(\lambda ).
\end{equation}

\begin{proposition}
\label{p2}For large values of $k$, the following representation holds%
\[
r_{k}(\lambda )=\frac{2h\sin (\alpha \pi )\lambda ^{-\alpha }\left[ \frac{%
\lambda ^{1/2}-\tau ^{1/2}}{\lambda ^{1/2}+\tau ^{1/2}}\right] ^{2k}}{\left(
\lambda ^{-\alpha }+h\right) ^{2}}\left( 1+\mathcal{O}\left( \frac{1}{k}%
\right) \right) +\mathcal{O}\left( \left( e_{k}(\lambda )\right) ^{2}\right) .
\]
\end{proposition}

\begin{proof}
By (\ref{efp}) we have%
\begin{equation}
\begin{split}
r_{k}(\lambda ) &=\frac{\lambda ^{-\alpha }}{\lambda ^{-\alpha }+h}-\frac{%
\mathcal{R}_{k-1,k}(\lambda )}{\mathcal{R}_{k-1,k}(\lambda )+h} \\
&=\frac{\lambda ^{-\alpha }}{\lambda ^{-\alpha }+h}-\frac{\lambda ^{-\alpha
}-e_{k}(\lambda )}{\lambda ^{-\alpha }-e_{k}(\lambda )+h} \\
&=\frac{he_{k}(\lambda )}{\left( \lambda ^{-\alpha }+h\right) \left(
\lambda ^{-\alpha }-e_{k}(\lambda )+h\right) } \\
&=\frac{he_{k}(\lambda )}{\left( \lambda ^{-\alpha }+h\right) ^{2}}+\mathcal{O} \left(
\left( e_{k}(\lambda )\right) ^{2}\right).
\end{split}
\end{equation}%
Therefore by Proposition \ref{p1} we find the result.
\end{proof}

In order to minimize  the error $E_k$ defined in (\ref{EK}), by Proposition~\ref{p2} the basic aim is now to study the nonnegative
function%
\begin{equation}
g_{k}(\lambda )=\frac{\lambda ^{-\alpha }\left[ \frac{\lambda ^{1/2}-\tau
^{1/2}}{\lambda ^{1/2}+\tau ^{1/2}}\right] ^{2k}}{\left( \lambda ^{-\alpha
}+h\right) ^{2}}  \label{gk}
\end{equation}%
and, in particular, to approximate the solution of%
\begin{equation}
\min_{\tau >0}\max_{\lambda \in \lbrack c,+\infty )}g_{k}(\lambda ).
\label{minimax}
\end{equation}

\begin{proposition} \label{p0}
The function $g_{k}(\lambda )$ given in (\ref{gk}) has the following
properties:

\begin{enumerate}
\item $g_{k}(\lambda )=0$ for $\lambda =\tau;$

\item $g_{k}(\lambda )\rightarrow 0$ for $\lambda \rightarrow 0^{+}$ and for 
$\lambda \rightarrow +\infty;$

\item $g_{k}(\lambda )$ has exactly two maximums $\lambda _{1}$ and $\lambda
_{2}$ such that%
\begin{equation*}
0<\lambda _{1} \lesssim \frac{\alpha ^{2}\tau }{4k^{2}},\quad \lambda_{2} \gtrsim \frac{4k^{2}\tau }{\alpha ^{2}}.
\end{equation*}
\end{enumerate}
\end{proposition}

\begin{proof}
Items (1) and (2) are obvious. As for item (3) the study of $\frac{d}{d\lambda }%
g_{k}(\lambda )=0$, after some algebra, leads to the equation%
\begin{equation}
\lambda ^{-\alpha }=h\frac{\alpha \left( 1-\frac{\tau }{\lambda }\right)
-2k\left( \frac{\tau }{\lambda }\right) ^{1/2}}{\alpha \left( 1-\frac{\tau }{%
\lambda }\right) +2k\left( \frac{\tau }{\lambda }\right) ^{1/2}}.  \label{ab}
\end{equation}%
The function on the right%
\begin{equation*}
d(\lambda ):=h\frac{\alpha \left( 1-\frac{\tau }{\lambda }\right) -2k\left( 
\frac{\tau }{\lambda }\right) ^{1/2}}{\alpha \left( 1-\frac{\tau }{\lambda }%
\right) +2k\left( \frac{\tau }{\lambda }\right) ^{1/2}}
\end{equation*}%
is the ratio of two parabolas in the variable $\lambda ^{1/2}$. Moreover $%
d(\lambda )\rightarrow h$ for $\lambda \rightarrow 0^{+}$ and for $\lambda
\rightarrow +\infty $, and it is not defined (in $[0,+\infty )$) at%
\begin{equation*}
\lambda ^{\ast }=\tau \left( \frac{-k+\sqrt{k^{2}+\alpha ^{2}}}{\alpha }%
\right) ^{2}\sim \frac{\alpha ^{2}\tau }{4k^{2}}.
\end{equation*}%
Moreover $d(\lambda )=0$ for%
\begin{equation*}
\lambda ^{\ast \ast }=\tau \left( \frac{k+\sqrt{k^{2}+\alpha ^{2}}}{\alpha }%
\right) ^{2}\sim \frac{4k^{2}\tau }{\alpha ^{2}}.
\end{equation*}%
Therefore starting from the point with coordinates $(0,h)$, $d(\lambda )$ is
growing and $d(\lambda ) \rightarrow +\infty $ for $\lambda \rightarrow \lambda ^{\ast -}.$
Moreover,  $d(\lambda ) \rightarrow -\infty $ for $\lambda \rightarrow \lambda ^{\ast +}$. From $%
\lambda ^{\ast }$ to $+\infty $ the function $d(\lambda )$ is still growing,
and $d(\lambda )<0$ for $\lambda \in \left( \lambda ^{\ast },\lambda ^{\ast
\ast }\right) $ and $d(\lambda )>0$ for $\lambda >\lambda ^{\ast \ast }$. As
consequence the equation (\ref{ab}) has exactly two solutions, $\lambda
_{1}<\lambda ^{\ast }$ and $\lambda _{2}>\lambda ^{\ast \ast }$.
\end{proof}

\begin{proposition}
\label{p4}For the maximum $\lambda _{2}$ it holds%
\begin{equation}
\lambda _{2}\sim \overline{\lambda }_{2}:=s_{k}\frac{4k^{2}\tau }{\alpha ^{2}%
},  \label{c0}
\end{equation}%
where%
\begin{equation}
s_{k}=\exp \left( \frac{1}{\alpha }W\left( \frac{4\alpha }{h\left( \frac{%
4k^{2}\tau }{\alpha ^{2}}\right) ^{\alpha }}\right) \right) ,  \label{ad}
\end{equation}%
and $W$ denotes the Lambert-W function.
\end{proposition}

\begin{proof}
Since $\lambda _{2}>\lambda ^{\ast \ast }>\frac{4k^{2}\tau }{\alpha ^{2}}$,
there exists $s>1$ such that%
\begin{equation}
\lambda _{2}=s\frac{4k^{2}\tau }{\alpha ^{2}}.  \label{l2}
\end{equation}%
Therefore 
\begin{equation*}
\left( \frac{\tau }{\lambda _{2}}\right) ^{1/2}=\frac{1}{\sqrt{s}}\frac{%
\alpha }{2k}.
\end{equation*}%
Using (\ref{ab}), for large values of $k$ we find%
\begin{equation}
\frac{\alpha \left( 1-\frac{\tau }{\lambda }\right) -2k\left( \frac{\tau }{%
\lambda }\right) ^{1/2}}{\alpha \left( 1-\frac{\tau }{\lambda }\right)
+2k\left( \frac{\tau }{\lambda }\right) ^{1/2}} =\frac{\alpha \left( 1-%
\frac{1}{s}\frac{\alpha ^{2}}{4k^{2}}\right) -\frac{\alpha }{\sqrt{s}}}{%
\alpha \left( 1-\frac{1}{s}\frac{\alpha ^{2}}{4k^{2}}\right) +\frac{\alpha }{%
\sqrt{s}}} \sim \frac{1-\frac{1}{\sqrt{s}}}{1+\frac{1}{\sqrt{s}}}.
\end{equation}%
Moreover, since the function $d(\lambda )$ is concave for $\lambda >\lambda
^{\ast \ast }$ and $d(\lambda )\rightarrow h$ as $\lambda \rightarrow
+\infty $ we have that $s\rightarrow 1$ as $k\rightarrow +\infty $, so that
we can use the approximation%
\begin{equation}
\frac{1-\frac{1}{\sqrt{s}}}{1+\frac{1}{\sqrt{s}}}\sim \frac{\ln s}{4}.
\label{qq}
\end{equation}%
By (\ref{ab}) and (\ref{l2}) we then have to solve%
\begin{equation}
\left( s\frac{4k^{2}\tau }{\alpha ^{2}}\right) ^{-\alpha }=h\frac{\ln s}{4},
\label{ss}
\end{equation}%
whose solution is given by (\ref{ad}).
\end{proof}

\begin{remark}
Since $W(x)=x+\mathcal{O}(x^{2})$ for $x$ close to $0$ (see, e.g., \cite[Eq. (3.1)]{CGHJK}), for $s_{k}$ defined in (\ref%
{ad}) we have%
\begin{equation}
s_{k}\sim \exp \left( \frac{4}{h}\left( \frac{\alpha ^{2}}{4k^{2}\tau }%
\right) ^{\alpha }\right) .  \label{ae}
\end{equation}
\end{remark}

\begin{proposition}
\label{p5} For the function $g_k$ defined in (\ref{gk}) it holds%
\begin{equation*}
g_{k}(\lambda _{2})\sim \frac{1}{h^{2}}\left( \frac{\alpha ^{2}}{4k^{2}\tau }%
\right) ^{\alpha }\exp (-2\alpha ).
\end{equation*}
\end{proposition}

\begin{proof}
Observe first that $s_{k}\rightarrow 1$ for $k\rightarrow +\infty $ (cf. (%
\ref{ad}) and (\ref{ae})), and hence by (\ref{c0})%
\begin{equation}
\begin{split}
\frac{\overline{\lambda }_{2}^{-\alpha }}{\left( \overline{\lambda }%
_{2}^{-\alpha }+h\right) ^{2}} &=\frac{1}{\left( \left( s_{k}\frac{%
4k^{2}\tau }{\alpha ^{2}}\right) ^{-\alpha }+h\right) ^{2}}\left( s_{k}\frac{%
4k^{2}\tau }{\alpha ^{2}}\right) ^{-\alpha } \\
&\sim \frac{1}{h^{2}}\left( \frac{\alpha ^{2}}{4k^{2}\tau }\right) ^{\alpha}.
\end{split}
\end{equation}%
Moreover,%
\begin{equation}
\left[ \frac{\overline{\lambda }_{2}^{1/2}-\tau ^{1/2}}{\overline{\lambda }%
_{2}^{1/2}+\tau ^{1/2}}\right] ^{2k} =\left[ \frac{1-\frac{\alpha }{2k%
\sqrt{s_{k}}}}{1+\frac{\alpha }{2k\sqrt{s_{k}}}}\right] ^{2k} \sim \exp \left( -\frac{2\alpha }{\sqrt{s_{k}}}\right) .
\end{equation}
The result then follows from (\ref{gk}).
\end{proof}

At this point we need to remember that our aim is to solve (\ref{minimax}).
By Proposition \ref{p0} we have that%
\begin{equation*}
\max_{\lambda \in \lbrack c,+\infty )}g_{k}(\lambda )=\max \left\{
g_{k}(\lambda _{1}),g_{k}(\lambda _{2})\right\} .
\end{equation*}%
Moreover, since $\lambda _{1}\rightarrow 0$ for $k\rightarrow +\infty $, for 
$k$ large enough we have%
\begin{equation*}
\max_{\lambda \in \lbrack c,+\infty )}g_{k}(\lambda )=\max \left\{
g_{k}(c),g_{k}(\lambda _{2})\right\} .
\end{equation*}%
Since we need to minimize the above quantity with respect to $\tau $, let us
consider the functions%
\begin{equation}
\varphi _{1}(\tau )=\frac{c^{-\alpha }\left[ \frac{c^{1/2}-\tau ^{1/2}}{%
c^{1/2}+\tau ^{1/2}}\right] ^{2k}}{\left( c^{-\alpha }+h\right) ^{2}}%
=g_{k}(c)  \label{ph}
\end{equation}%
and%
\begin{equation*}
\varphi _{2}(\tau )=\left( \frac{\alpha ^{2}}{4k^{2}\tau }\right) ^{\alpha }%
\frac{1}{h^{2}}\exp (-2\alpha )\sim g_{k}(\lambda _{2}).
\end{equation*}

It is easy to see that $\varphi _{1}(\tau )$ is monotone increasing for $%
\tau >c$, whereas $\varphi _{2}(\tau )$ is monotone decreasing. Therefore,
for $k$ large enough, the exact solution $\widetilde{\tau }$ of (\ref%
{minimax}) can be approximated by solving $\varphi _{1}(\tau )=\varphi
_{2}(\tau )$.

\begin{proposition}
\label{p6} Let 
\begin{equation}
\phi _{k}=\frac{\alpha}{2ke} \left( \frac{c^{-\alpha }+h}{h}\right)
^{1/\alpha}.  \label{phik}
\end{equation}
For $k$ large enough the solution of $\varphi _{1}(\tau )=\varphi _{2}(\tau
) $ is approximated by%
\begin{equation}
\tau _{k}:=c \, \phi _{k}^2 \exp \left( 2W\left( \frac{2k}{\phi _{k} \alpha }%
\right) \right).  \label{tauk}
\end{equation}
\end{proposition}

\begin{proof}
By (\ref{phik}), the equation $\varphi _{1}(\tau )=\varphi _{2}(\tau )$
leads to%
\begin{equation}
\left( \frac{c}{\tau }\right) ^{-\alpha }\left[ \frac{c^{1/2}-\tau ^{1/2}}{%
c^{1/2}+\tau ^{1/2}}\right] ^{2k}=\phi _{k}^{2\alpha }.  \label{eq1}
\end{equation}%
Since the exact solution of (\ref{eq1}) goes to infinity with $k$, we set%
\begin{equation}
x:=\left( \frac{c}{\tau }\right) ^{1/2}<1,  \label{rt}
\end{equation}%
so that by (\ref{eq1}) we obtain%
\begin{equation}
x^{-2\alpha }\left[ \frac{1-x}{1+x}\right] ^{2k}=\phi _{k}^{2\alpha }.
\label{eq2}
\end{equation}%
Since $(1+x)^{-1} = 1-x + \mathcal{O}{(x^2),}$ using the approximation%
\begin{equation}
\left[ \frac{1-x}{1+x}\right] ^{2k}\sim \exp \left( -4kx\right) 
\label{apprexp}
\end{equation}%
we then want to solve%
\begin{equation*}
\exp \left( -4kx\right) =(\phi _{k}x)^{2\alpha },
\end{equation*}%
whose solution, approximation to the one of (\ref{eq2}), is given by%
\begin{equation*}
\overline{x}:=\left[ \phi _{k}\exp \left( W\left( \frac{2k}{\phi _{k}\alpha }%
\right) \right) \right] ^{-1}.
\end{equation*}%
The result then follows immediately from (\ref{rt}).
\end{proof}

In Figure~\ref{fig1} we consider the graphical interpretation of the analysis that leads us to the definition of $\tau_k$ in Proposition~\ref{p6}. Assuming to work with a spectrum contained in $[c, +\infty),$ with $c=1,$ we define $\tau_k$ by solving $g_k(c)=g_k(\lambda_2).$ As already pointed out, the leftmost maximum $\lambda_1$ becomes smaller than $c$ for $k$ large enough.

\begin{figure}
% file usato:       fungk.m    
\includegraphics[width=0.90\textwidth]{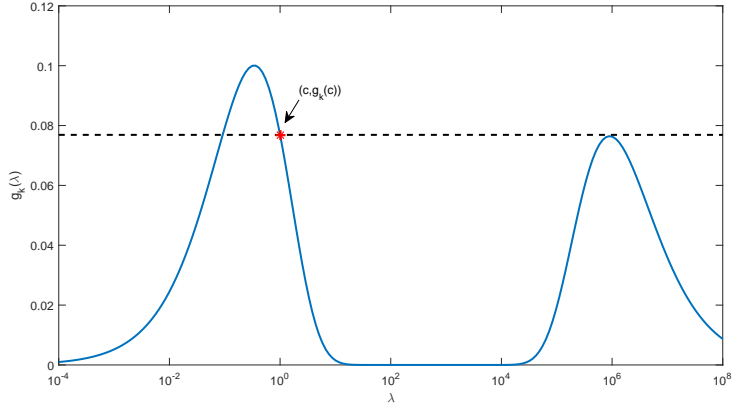}
\caption{Behavior of the function $g_k$ defined in (\ref{gk}) versus $\lambda$ (in logarithmic scale) for $\alpha=0.75, k=15, h=10^{-2},$  $\tau_k$  as defined in (\ref{tauk}) and $c=1.$}
\label{fig1}
\end{figure}

Finally, we are on the point to give the following result, that provides an error estimate
since (see (\ref{EK}) and (\ref{errek}))
\[
\left\Vert \left( I+h\mathcal{L}^{\alpha }\right) ^{-1}-\mathcal{S}%
_{k-1,k}(\mathcal{L})\right\Vert _{\mathcal{H}\rightarrow \mathcal{H}} \le \max_{\lambda \in \lbrack c,+\infty )} |r_{k}(\lambda )|.
\]

\begin{theorem}
\label{th1} Let $\tau _{k}$ be defined according to (\ref{tauk}). Then for $k
$ large enough%
\begin{equation*}
\max_{\lambda \in \lbrack c,+\infty )}|r_{k}(\lambda )|\sim \frac{2\sin
(\alpha \pi )c^{-\alpha }}{h}\left( \frac{2ke^{1/2}}{\alpha }\right)
^{-4\alpha }\left( \ln \left[ \frac{4k^{2}e}{\alpha ^{2}}\left( \frac{h}{%
c^{-\alpha }+h}\right) ^{\frac{1}{\alpha }}\right] \right) ^{2\alpha }.
\end{equation*}
\end{theorem}

\begin{proof}
By Proposition \ref{p2} and (\ref{gk}) we have that%
\begin{equation*}
|r_{k}(\lambda )| \sim 2h\sin (\alpha \pi )g_{k}(\lambda ).
\end{equation*}%
Then using Proposition \ref{p6}, that is, taking $\tau =\tau _{k}$ as in (\ref{tauk}), we have%
\begin{equation}  \label{erc}
\max_{\lambda \in \lbrack c,+\infty )}|r_{k}(\lambda )| \sim 2h\sin (\alpha
\pi )\max_{\lambda \in \lbrack c,+\infty )}g_{k}(\lambda )   \sim 2h\sin (\alpha \pi )\varphi _{2}(\tau _{k}). 
\end{equation}
Since for large $z$%
\begin{equation*}
W(z)=\ln z-\ln (\ln z)+\mathcal{O}\left( \frac{\ln (\ln z)}{\ln z}\right) ,
\end{equation*}%
cf. \cite{HH}, we have that%
\begin{equation*}
\exp \left( 2W\left( z\right) \right) \sim \frac{z^{2}}{(\ln z)^{2}},
\end{equation*}%
and hence 
\begin{equation*}
\tau _{k}\sim c\,\phi _{k}^{2}\left( \frac{2k}{\phi _{k}\alpha }\right)
^{2}\left( \ln \frac{2k}{\phi _{k}\alpha }\right) ^{-2}.
\end{equation*}%
By inserting this approximation in (\ref{erc}) we obtain the result.
\end{proof}

%%%%%%%%%%%%%%%%%%%%%%%%%%%%%%%%%%%%%%%%%%%
%%%%%%%%%%%%%%%%%%%%%%%%%%%%%%%%%%%%%%%%%%%

\section{The case of bounded operators}

\label{sec4} 

Let $\mathcal{L}_{N}$ be an arbitrary sharp discretization of $%
\mathcal{L}$ with spectrum contained in $[c,\lambda _{N}],$ where $\lambda
_{N}$ denotes the largest eigenvalue of $\mathcal{L}_{N}.$ The theory just
developed can easily be adapted to the approximation of $\left( I+h\mathcal{L%
}_{N}^{\alpha }\right) ^{-1}.$ In this situation, in order to define a
nearly optimal value for $\tau $, similarly to (\ref{minimax}) we want to
solve%
\begin{equation}
\min_{\tau >0}\max_{c\leq \lambda \leq \lambda _{N}}g_{k}(\lambda ).
\label{minmax2}
\end{equation}

Looking at Proposition \ref{p4} we have $\lambda _{2}=\lambda
_{2}(k)\rightarrow +\infty $ as $k\rightarrow +\infty $. As a consequence,
for $\lambda _{2}\leq \lambda _{N}$ ($k$ small), the solution of (\ref%
{minmax2}) remains the one approximated by (\ref{tauk}) and the estimate
given in Theorem \ref{th1} is still valid. On the contrary, for $\lambda
_{2}>\lambda _{N}$ ($k$ large), the estimate can be improved as follows.
Remembering the features of the function $g_{k}(\lambda )$  given in Proposition \ref{p0}, 
we have that for $\lambda _{2}>\lambda _{N}$ the solution of (\ref%
{minmax2}) is obtained by solving%
\begin{equation}
\varphi _{1}\left( \tau \right) =\varphi _{3}\left( \tau \right) \mbox{ for }%
\tau >c,  \label{pbt2}
\end{equation}%
where $\varphi _{1}\left( \tau \right) $ is defined in (\ref{ph}) and 
\begin{equation*}
\varphi _{3}\left( \tau \right) :=\frac{\lambda _{N}^{-\alpha }\left[ \frac{%
\lambda _{N}^{1/2}-\tau ^{1/2}}{\lambda _{N}^{1/2}+\tau ^{1/2}}\right] ^{2k}%
}{\left( \lambda _{N}^{-\alpha }+h\right) ^{2}}=g_{k}(\lambda _{N}).
\end{equation*}%
It can be easily verified that the equation $\varphi _{1}\left( \tau \right)
=\varphi _{3}\left( \tau \right) $ has in fact two solutions, one in the
interval $(0,c)$ and the other in $(c,\lambda _{N})$. Anyway since $\varphi
_{3}\left( \tau \right) $ is monotone decreasing in $[0,\lambda _{N})$ we
have to look for the one in $(c,\lambda _{N})$ as stated in (\ref{pbt2}).

\begin{proposition}
For $k$ large enough, the solution of (\ref{pbt2}) is approximated by 
\begin{equation}
\tau _{k,N}:={\left( -\sigma _{k}+\sqrt{\sigma _{k}^{2}+\left( c\,\lambda
_{N}\right) ^{1/2}}\right) ^{2}},  \label{tk2}
\end{equation}%
where 
\begin{equation*}
\sigma _{k}:=\frac{\alpha \lambda _{N}^{1/2}}{8k}\ln \left( \frac{\lambda
_{N}}{c}\left( \frac{\lambda _{N}^{-\alpha }+h}{c^{-\alpha }+h}\right)
^{2/\alpha }\right) .
\end{equation*}
\end{proposition}

\begin{proof}
From (\ref{pbt2}) we have 
\begin{equation}
\frac{c^{-\alpha }}{(c^{-\alpha }+h)^{2}}\left[ \frac{\tau ^{1/2}-c^{1/2}}{%
\tau ^{1/2}+c^{1/2}}\right] ^{2k}=\frac{\lambda _{N}^{-\alpha }}{(\lambda
_{N}^{-\alpha }+h)^{2}}\left[ \frac{\lambda _{N}^{1/2}-\tau ^{1/2}}{\lambda
_{N}^{1/2}+\tau ^{1/2}}\right] ^{2k}.  \label{pr1}
\end{equation}%
Setting $x=\left( {c}/{\tau }\right) ^{1/2}<1$ and $y=\left( \tau /\lambda
_{N}\right) ^{1/2}<1$ by (\ref{pr1}) we obtain 
\begin{equation*}
\left( \frac{1-x}{1+x}\right) =\left( \frac{\lambda _{N}}{c}\right) ^{-\frac{%
\alpha }{2k}}\left( \frac{c^{-\alpha }+h}{\lambda _{N}^{-\alpha }+h}\right)
^{\frac{1}{k}}\left( \frac{1-y}{1+y}\right) .
\end{equation*}%
Using (\ref{apprexp}) we solve%
\begin{equation*}
e^{-2x}=\left( \frac{\lambda _{N}}{c}\right) ^{-\frac{\alpha }{2k}}\left( 
\frac{c^{-\alpha }+h}{\lambda _{N}^{-\alpha }+h}\right) ^{\frac{1}{k}%
}e^{-2y}.
\end{equation*}%
Therefore 
\begin{equation*}
-2x=-\frac{\alpha }{2k}\ln \left( \frac{\lambda _{N}}{c}\left( \frac{\lambda
_{N}^{-\alpha }+h}{c^{-\alpha }+h}\right) ^{2/\alpha }\right) -2y,
\end{equation*}%
which implies 
\begin{equation*}
x-y=\frac{\alpha }{4k}\ln \left( \frac{\lambda _{N}}{c}\left( \frac{\lambda
_{N}^{-\alpha }+h}{c^{-\alpha }+h}\right) ^{2/\alpha }\right) .
\end{equation*}%
Substituting $x$ by $\left( {c}/{\tau }\right) ^{1/2}$ and $y$ by $\left(
\tau /\lambda _{N}\right) ^{1/2}$ after some algebra we obtain 
\begin{equation*}
\tau +\frac{\alpha }{4k}\lambda _{N}^{1/2}\ln \left( \frac{\lambda _{N}}{c}%
\left( \frac{\lambda _{N}^{-\alpha }+h}{c^{-\alpha }+h}\right) ^{2/\alpha
}\right) \tau ^{1/2}-(c\,\lambda _{N})^{1/2}=0.
\end{equation*}%
Then, solving this equation and taking the positive solution, we obtain the
expression of $\tau _{k,N}.$
\end{proof}

Observe that by (\ref{tk2}), for $k\rightarrow +\infty $ we have 
\begin{equation*}
\left( \frac{\tau _{k,N}}{\lambda _{N}}\right) ^{1/2}=-\frac{\alpha }{8k}\ln
\left( \frac{\lambda _{N}}{c}\left( \frac{\lambda _{N}^{-\alpha }+h}{%
c^{-\alpha }+h}\right) ^{2/\alpha }\right) +\left( \frac{c}{\lambda _{N}}%
\right) ^{1/4}+\mathcal{O}\left( \frac{1}{k^{2}}\right) .
\end{equation*}%
Moreover, using (\ref{apprexp}) and the above expression we obtain%
\begin{equation}
\begin{split}
\varphi _{3}\left( \tau _{k,N}\right)  &=\frac{\lambda _{N}^{-\alpha }}{%
(\lambda _{N}^{-\alpha }+h)^{2}}\left[ \frac{\lambda _{N}^{1/2}-\tau _{k,N}}{%
\lambda _{N}^{1/2}+\tau _{k,N}}\right] ^{2k}  \notag \\
&\sim \frac{\lambda _{N}^{-\alpha }}{(\lambda _{N}^{-\alpha }+h)^{2}}\exp
\left( -4k\left( \frac{\tau _{k,N}}{\lambda _{N}}\right) ^{1/2}\right)  
\notag \\
&\sim \frac{\left( c\,\lambda _{N}\right) ^{-\alpha /2}}{(c^{-\alpha
}+h)(\lambda _{N}^{-\alpha }+h)}\exp \left( -4k\left( \frac{c}{\lambda _{N}}%
\right) ^{1/4}\right) ,  \label{res3}
\end{split}
\end{equation}
that proves the following result.

\begin{theorem}
\label{MT2}Let $\overline{k}$ be such that for each $k\geq \overline{k}$ we
have $\lambda _{2}=\lambda _{2}(k)>\lambda _{N}.$ Then for each $k\geq 
\overline{k}$, taking in (\ref{jss}) $\tau =\tau _{k,N},$ where $\tau _{k,N}$
is given in (\ref{tk2}), the following estimate holds%
\begin{equation}
\begin{split}
\left\Vert \left( I+h\mathcal{L}_{N}^{\alpha }\right) ^{-1}-%
\mathcal{S}_{k-1,k}(\mathcal{L}_{N})\right\Vert _{2} &\sim
2h\sin (\alpha \pi )\frac{\left( c\,\lambda _{N}\right) ^{-\alpha /2}}{%
(c^{-\alpha }+h)(\lambda _{N}^{-\alpha }+h)}  \notag \\
&\times \exp \left( -4k\left( \frac{c}{\lambda _{N}}\right) ^{1/4}\right),  \label{th2}
\end{split}
\end{equation}
with $\Vert \cdot \Vert _{2}$ denoting the induced Euclidean norm.
\end{theorem}

In order to compute a fairly accurate estimate of $\overline{k}$ we need to
solve the equation $\lambda _{2}=\lambda _{N},$ where $\lambda _{2}$ is
defined in Proposition \ref{p4}. Neglecting the factor $s_{k}$ in (\ref{ad})
and taking $\tau =\tau _{k}$ as in (\ref{tauk}), we obtain the equation 
\begin{equation*}
W\left( \frac{2k}{\phi _{k}\alpha }\right) =\frac{1}{2}\ln \left( \frac{%
\lambda _{N}}{c}e^{2}\left( \frac{h}{c^{-\alpha }+h}\right) ^{2/\alpha
}\right) .
\end{equation*}%
Since $W(z_{1})=z_{2}$ if and only if $z_{1}=z_{2}e^{z_{2}},$ we have 
\begin{equation*}
\frac{4k^{2}}{\alpha ^{2}}=\frac{1}{2}\ln \left( \frac{\lambda _{N}}{c}%
e^{2}\left( \frac{h}{c^{-\alpha }+h}\right) ^{2/\alpha }\right) \left( \frac{%
\lambda _{N}}{c}\right) ^{1/2}
\end{equation*}%
from which it easily follows that 
\begin{equation*}
{\bar{k}}=\frac{\alpha }{2\sqrt{2}}\left( \ln \left( \frac{\lambda _{N}}{c}%
e^{2}\left( \frac{h}{c^{-\alpha }+h}\right) ^{2/\alpha }\right) \right)
^{1/2}\left( \frac{\lambda _{N}}{c}\right) ^{1/4}.
\end{equation*}

In practice, assuming to have a good estimate of the interval containing the
spectrum of $\mathcal{L}_{N},$ one should use $\tau _{k}$ as in (\ref{tauk})
whenever $k<\overline{k}$ and then switch to $\tau _{k,N}$ as in (\ref{tk2})
for $k\geq \overline{k}.$ In other words, for bounded operators we consider
the sequence 
\begin{equation}
\tau _{k,N}=\left\{ 
\begin{array}{lll}
\tau _{k} &  & \quad \mbox{if }k<\overline{k}, \\ 
\tau _{k,N} &  & \quad \mbox{if }k\geq \overline{k}.%
\end{array}%
\right.  \label{tauk2}
\end{equation}

%%%%%%%%%%%%%%%%%%%%%%%%%%%%%%%%%%%%%%%%%%%
%%%%%%%%%%%%%%%%%%%%%%%%%%%%%%%%%%%%%%%%%%%

\section{Numerical experiments}

\label{sec5}

In this section we present the numerical results obtained by considering two
simple cases of self-adjoint positive operators. The first one is totally
artificial since we just consider a diagonal matrix with a large spectrum.
In the second one we consider the standard central difference discretization of
the one dimensional Laplace operator with Dirichlet boundary conditions.

We remark that in all the experiments the weights and nodes of the
Gauss-Jacobi quadrature rule are computed by using the Matlab function 
\texttt{jacpts} implemented in Chebfun by Hale and Townsend \cite{HT}. In
addition, the errors are always plotted with respect to the Euclidean norm.

\begin{example}
\label{e1} We define $A=\func{diag}(1,2,\dots ,N)$ and $\mathcal{L}_{N}=A^{p}
$ so that $\sigma (\mathcal{L}_{N})\subseteq \lbrack 1,N^{p}]$. Taking $%
N=100,$ $p=7,$ and $h=10^{-2},$ in Figure \ref{fig2}, for $\alpha
=0.2,0.4,0.6,0.8$ we plot the error obtained using $\tau _{k}$ taken as in (%
\ref{tauk}) and $\tilde{\tau}_{k}$ as defined in \cite[Eq. (24)]{AN0}, that
is,%
\begin{equation}
\tilde{\tau}_{k}:=c\left( \frac{\alpha }{2ke}\right) ^{2}\exp \left(
2W\left( \frac{4k^{2}e}{\alpha ^{2}}\right) \right) .  \label{an0}
\end{equation}%
In addition, we draw the values of the estimate given in Theorem \ref{th1}.

\begin{figure}
% file usato:     meth_inf_res.m   
\includegraphics[width=1.00\textwidth]{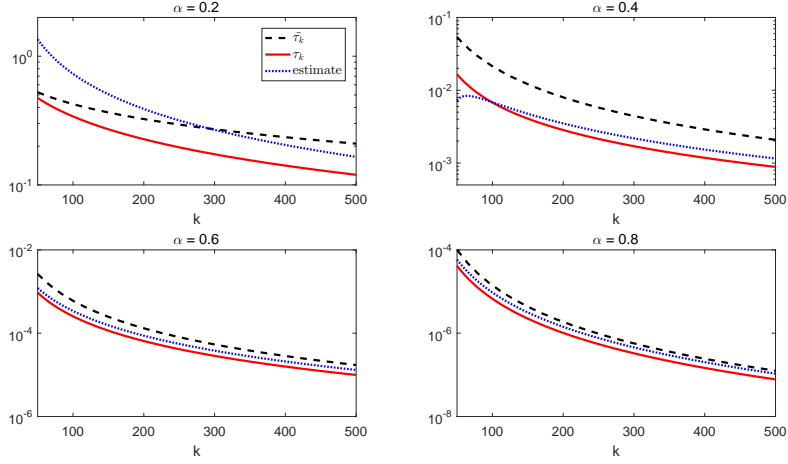}
\caption{Error comparison using $\protect\tau _{k}$ taken as in (\protect\ref%
{tauk}) (solide line) and $\tilde{\protect\tau _{k}}$ as defined in 
(\ref{an0}) (dashed line) for Example 1 with $N=100,$ $p=7,$
and $h=10^{-2}$. The dotted line represents the values of the estimate
given in Theorem \protect\ref{th1}. }
\label{fig2}
\end{figure}

In Figure \ref{fig3} we consider the choice of $\tau =\tau _{k,N}$ as in (%
\ref{tauk2}) since we take $p=3$, that is, an operator with a moderately
large spectrum. We compare this choice with the analogous one proposed in  
\cite[Eq. (37)]{AN0} and given by%
\begin{equation}
\widetilde{\tau }_{k,N}:={\left( -\frac{\alpha \lambda _{N}^{1/2}}{8k}\ln
\left( \frac{\lambda _{N}}{c}\right) +\sqrt{\left( \frac{\alpha \lambda
_{N}^{1/2}}{8k}\ln \left( \frac{\lambda _{N}}{c}\right) \right) ^{2}+\left(
c\,\lambda _{N}\right) ^{1/2}}\right) ^{2}.}  \label{anN}
\end{equation}%
In the pictures we also plot the error estimate (\ref{th2}).

\begin{figure}
% file usato:    meth_bnd_res.m   
\includegraphics[width=1.00\textwidth]{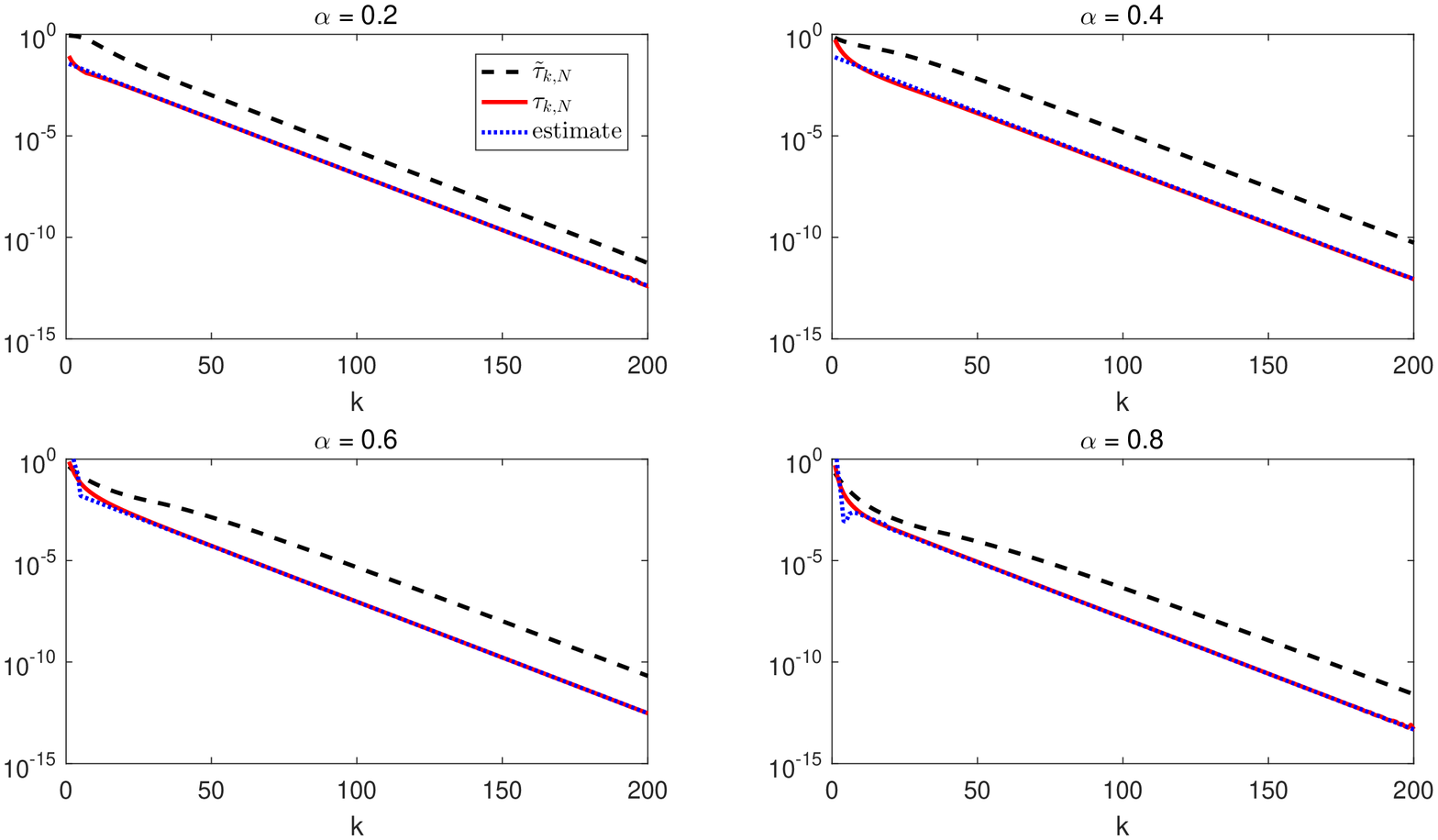}
\caption{Error comparison using $\protect\tau =\protect\tau _{k,N}$ as in (%
\protect\ref{tauk2}) (solide line) and $\tilde{\protect\tau}_{k,N}$ as
defined in (\ref{anN}) (dashed line) for Example \protect
\ref{e1} with $N=100,$ $p=3,$ and $h=10^{-2}.$ In the pictures we also plot
the error estimate (\protect\ref{th2}) (dotted line).}
\label{fig3}
\end{figure}
\end{example}

\begin{example}
\label{e2} We consider the linear operator $\mathcal{L}u=-u^{\prime \prime}, 
$ $u:[0, b]\rightarrow \mathbb{R},$ with Dirichlet boundary conditions $%
u(0)=u(b)=0$. It is known that $\mathcal{L}$ has a point spectrum consisting
entirely of eigenvalues 
\begin{equation*}
\mu_{s}=\frac{\pi ^{2}s^{2}}{b^{2}},\qquad \mbox{for }s=1,2,3,\dots.
\end{equation*}

Using the standard central difference scheme on a uniform grid and setting $%
b=1$, in this example we work with the operator 
\begin{equation}
\mathcal{L}_{N}:=(N+1)^{2}\func{tridiag}(-1,2,-1)\in \mathbb{R}^{N\times N}.
\label{LN}
\end{equation}%
The eigenvalues are 
\begin{equation*}
\lambda _{j}=4(N+1)^{2}\sin ^{2}\left( \frac{j\pi }{2(N+1)}\right) ,\qquad
j=1,2,\dots ,N,
\end{equation*}%
so that $\sigma (\mathcal{L}_{N})\subseteq \lbrack \pi ^{2},4(N+1)^{2}].$

Taking $N=1000$ and $h=10^{-2}$, in Figure \ref{fig4}, for $\alpha =0.6$ we
plot the error obtained using $\tau _{k,N}$ taken as in (\ref{tauk2}) and $%
\tilde{\tau}_{k,N}$ as in (\ref{anN}).

\begin{figure}
% file usato:       meth_bnd_lap_res.m    
\includegraphics[width=1.00\textwidth]{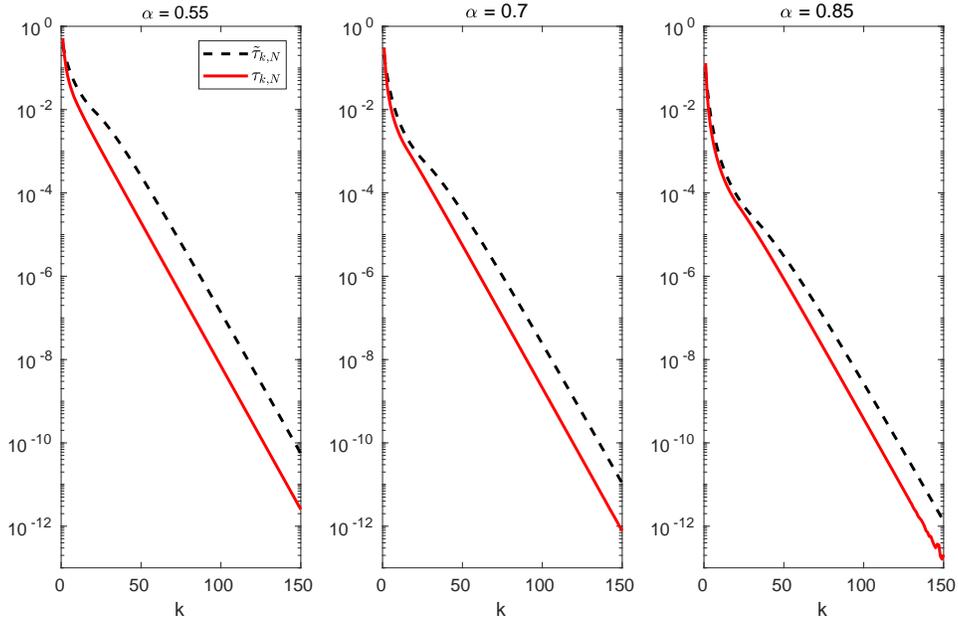}
\caption{Error comparison using $\protect\tau _{k,N}$ taken as in (\protect
\ref{tauk2}) (solide line) and $\tilde{\protect\tau}_{k,N}$ as defined in (%
\protect\ref{anN}) (dashed line) for Example \protect\ref{e2} with $N=1000,$
and $h=10^{-2}.$}
\label{fig4}
\end{figure}
\end{example}

\begin{example}
\label{e3} In this final example we want to consider the use of the poles
arising from the rational approximation introduced in Section \ref{sec2} for
the construction of rational Krylov methods (RKM), see e.g. \cite{BR, DKZ, Gu}. In this view, let%
\begin{equation*}
\mathcal{W}_{k}(\mathcal{L}_{N},v)=\func{Span}\{v,(\overline{\eta }_{1}I+%
\mathcal{L}_{N})^{-1}v,\ldots ,(\overline{\eta }_{1}I+\mathcal{L}_{N})^{-1} \cdots(\overline{\eta }_{k-1}I+\mathcal{L}_{N})^{-1}v\},
\end{equation*}%
be the $k$-dimensional rational Krylov subspace in which $\{\overline{\eta }%
_{1},\ldots ,\overline{\eta }_{k-1}\}$ is the set of abscissas as in (\ref%
{npad}), $\mathcal{L}_{N}$ defined by (\ref{LN}), and $v$ is a given vector.
Denoting by $V_{k}$ the orthogonal matrix whose columns span $\mathcal{W}%
_{k}(\mathcal{L}_{N},v)$ we consider the rational Krylov approximation%
\begin{equation}
\omega_{k}:=V_{k}\left( I+hH_{k}^{\alpha }\right) ^{-1}V_{k}^{T}v\approx \left(
I+h\mathcal{L}_{N}^{\alpha }\right) ^{-1}v,  \label{kapp}
\end{equation}%
in which $H_{k}=V_{k}^{T}\mathcal{L}_{N}V_{k}$. We remark that $\left( I+h%
\mathcal{L}_{N}^{\alpha }\right) ^{-1}v$ is just the result of one step of
length $h$ of the implicit Euler method applied to the discrete fractional
diffusion problem%
\begin{equation*}
y^{\prime }=\mathcal{L}_{N}^{\alpha }y,\quad y(0)=v.
\end{equation*}%
By taking $\tau _{k}$ as in (\ref{tauk}) to define the set $\{\overline{\eta 
}_{1},\ldots ,\overline{\eta }_{k-1}\}$, in Figure \ref{fig5} we consider
the error %(with respect to the Euclidean vector norm) 
of the approxomation (\ref{kapp}), for $h=10^{-2}$ and $v$ corresponding to the
discretization of the scalar function $v(x)=x(1-x)$, for $x\in \lbrack 0,1].$ Since the 
construction of the Krylov subspace of dimension $k$ requires the
knowledge of the whole set $\{\overline{\eta }_{1},\ldots ,\overline{\eta }%
_{k-1}\}$, for $k=10,15,\dots ,30$ we compute the corresponding set and
consider the final Krylov approximation. In order to appreciate the quality
of the approximation we compare this approach with the analogous one in
which the set of shifts arises from $\widetilde{\tau }_{k}$ as in (\ref{an0}%
), and also with respect to the shift-and-invert Krylov method (SIKM), in
which we take $\overline{\eta }_{1}=\ldots =\overline{\eta }%
_{k-1}=h^{-1/\alpha }$, following the analysis given in \cite{Mo}.

We remark that for practical purposes one should be able to a-priori set the
dimension of the Krylov subspace and this of course requires an accurate
error estimate. In this view, the estimate given in Theorem \ref{th1} can be
used to this purpose, since (cf. \cite[Corollary 3.4]{Gu})%
\begin{equation}
\left\Vert \left( I+h\mathcal{L}_{N}^{\alpha }\right)
^{-1}v-\omega_{k}\right\Vert _{2}\leq 2\min_{p_{k}\in \Pi _{k}}\max_{c\leq
\lambda \leq \infty }\left\vert \left( 1+h\lambda ^{\alpha }\right) ^{-1}-%
\frac{p_{k}(\lambda )}{q_{k}(\lambda )}\right\vert \left\Vert v\right\Vert_2.
\label{kest}
\end{equation}%
Anyway we have to point out that using Theorem \ref{th1} in (\ref{kest}) the
resulting bound may be much conservative for two main reasons. The first one
is that we are in fact considering a $(k-1,k)$ approximation. The second one
is that Theorem \ref{th1} provides an estimate for general unbounded
operator whereas the Krylov method is tailored on the initial vector $v$ and
also depends on the eigenvalue distribution. For these reasons a practical
hint can be to define $k$ at the beginning using Theorem \ref{th1} and then
monitor the quality of the approximation at each Krylov iteration $j\leq k$
by means of the generalized residual given by%
\begin{equation*}
v_{j+1}^{T}\mathcal{L}_{N}v_{j}\left\vert e_{j}^{T}\left( I+hH_{j}^{\alpha
}\right) ^{-1}V_{j}^{T}v\right\vert,
\end{equation*}%
where $v_{j}$, $j=1,\dots ,k$, are the columns of $V_{k}$ and $%
e_{j}=(0,\dots,0,1)^{T}\in \mathbb{R}^{j}$.

\begin{figure}
% file usato:       driver_krylov_jacobi2.m    
\includegraphics[width=1.00\textwidth]{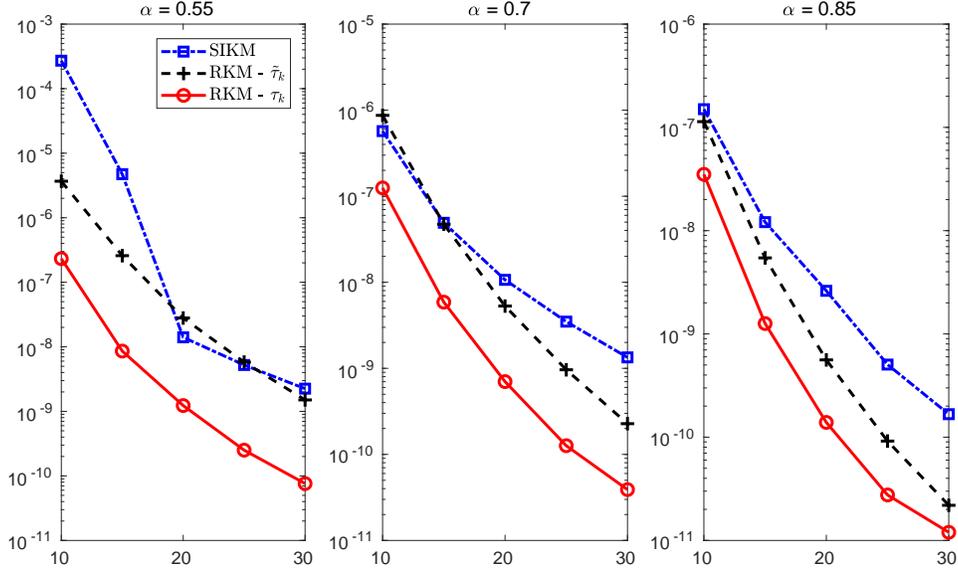}
\caption{Error comparison between SIKM, RKM with $\tilde{\protect\tau}_{k}$
as in (\protect\ref{an0}), RKM with $\protect\tau _{k}$ as in (\protect\ref%
{tauk}) for $\mathcal{L}_{N}$ defined in (\protect\ref{LN}), with $N=3000,$
and $h=10^{-2}$.}
\label{fig5}
\end{figure}
\end{example}

\section{Concluding remarks}

In this paper we have presented a reliable $\left( k-1,k\right) $ rational
approximation for the function $\left( 1+h\lambda ^{\alpha }\right) ^{-1}$
on a positive unbounded interval, that can be fuitfully used to compute
the resolvent of the fractional power in both the infinite and finite dimensional
setting. Moreover the theory can also be employed for the construction of
rational Krylov methods, with very good results. With respect to the simple
use of rational approximations to $\lambda ^{-\alpha }$, extended to compute 
$\left( 1+h\lambda ^{\alpha }\right) ^{-1}$ by means of (\ref{eqa}), in this
work we have shown that allowing a dependence on $h$ it is possible to
improve the quality of the approximation. We have provided sharp error
estimates that can be used for the a-priori choice of the number of poles,
that is, the number of inversions.

Remaining in the framework of Pad\'{e}-type approximations, we
want to point out that many other strategies are possible. Among the others
we present here two of them already tested experimentally.

\begin{enumerate}
\item Writing%
\begin{equation*}
\frac{1}{1+h\lambda ^{\alpha }}=\frac{1}{1+h\lambda \lambda ^{\alpha -1}},
\end{equation*}%
we can consider the Pad\'{e}-type approximation (\ref{pad}), with $-\alpha $
replaced by $\alpha -1$. In this way we obtain the approximation%
\begin{equation*}
\begin{split}
\frac{1}{1+h\lambda ^{\alpha }} &\approx \frac{1}{1+h\lambda \mathcal{R}%
_{k-1,k}(\lambda )} \\
&=\frac{q_{k}\left( \lambda \right) }{q_{k}\left( \lambda \right) +h\lambda
p_{k-1}\left( \lambda \right) },
\end{split}
\end{equation*}
that in fact represents a $(k,k)$ form. Unfortunately this approach is
observed to be in general less effective than the one defined in (\ref{sk})-(%
\ref{npad}).

\item Let $R_{k,k}(\lambda /\tau )$ be the $(k,k)$-Pad\'{e} approximant of $%
(\lambda /\tau )^{-\alpha }$ centered at $1$, whose error representation
with respect to the variable $z=1-\lambda /\tau $ has been derived in \cite[%
Section 3]{E}. As in (\ref{pad}) we can consider the formula 
\begin{equation*}
\lambda ^{-\alpha }\approx \mathcal{R}_{k,k}(\lambda ),\quad \mathcal{R}%
_{k,k}(\lambda ):=\tau ^{-\alpha }R_{k,k}(\lambda /\tau ),
\end{equation*}%
that yields the $(k,k)$ rational approximation%
\begin{equation*}
\frac{1}{1+h\lambda ^{\alpha }}\approx \frac{p_{k}(\lambda )}{q_{k}(\lambda
)+hp_{k}(\lambda )}=:\mathcal{S}_{k,k}(\lambda ),
\end{equation*}%
in which $p_{k},q_{k}\in \Pi _{k}$ are such that $\mathcal{R}_{k,k}(\lambda
)=p_{k}(\lambda )/q_{k}(\lambda )$. Using this approach, the relationship
with the Gauss-Jacobi rule explained in Section \ref{sec2} is lost.
Nevertheless, the error analysis is identical to the one given in Section %
\ref{sec3}, since the representation (\ref{ek}) is still valid with $k$
replaced by $k+1$. Also experimentally, this approach is almost identical to
the one presented in the paper.
\end{enumerate}

%%%%%%%%%%%%%%%%%%%%%%%%%%%%%%%%%%%%%%%%%%%
%%%%%%%%%%%%%%%%%%%%%%%%%%%%%%%%%%%%%%%%%%%

\bibliographystyle{amsplain}

\end{document}